\documentclass{amsart}
\makeatletter
\@namedef{subjclassname@2020}{%
  \textup{2020} Mathematics Subject Classification}
\makeatother 

\usepackage{amsthm,amssymb,amsfonts,latexsym,mathtools,thmtools}
\usepackage[T1]{fontenc}
\usepackage{tikz-cd} 
\usepackage{enumitem} 
\usepackage{hyperref} 
\usepackage{hyperref}
\hypersetup{
    colorlinks=true,
    linkcolor=blue,
    filecolor=blue,      
    urlcolor=cyan,
    linktocpage=true
}

\newtheorem{theorem}{Theorem}[section]

\theoremstyle{definition}
\newtheorem{definition}[theorem]{Definition}
\newtheorem{example}[theorem]{Example}

\theoremstyle{remark}
\newtheorem{remark}[theorem]{Remark}

\newtheorem{proposition}[theorem]{Proposition}
\newtheorem{corollary}[theorem]{Corollary}

\numberwithin{equation}{section}

\begin{document}

\title[On NI and NJ skew PBW extensions]{On NI and NJ skew PBW extensions}



\author{H\'ector Su\'arez}
\address{Universidad Pedag\'ogica y Tecnol\'ogica de Colombia, Sede Tunja}
\curraddr{Campus Universitario}
\email{hector.suarez@uptc.edu.co}
\thanks{}

\author{Andr\'es Chac\'on}
\address{Universidad Nacional de Colombia - Sede Bogot\'a}
\curraddr{Campus Universitario}
\email{anchaconca@unal.edu.co}
\thanks{}

\author{Armando Reyes}
\address{Universidad Nacional de Colombia - Sede Bogot\'a}
\curraddr{Campus Universitario}
\email{mareyesv@unal.edu.co}

\thanks{The authors were supported by the research fund of Faculty of Science, Code HERMES 52464, Universidad Nacional de Colombia - Sede Bogot\'a, Colombia.}



\date{}

\dedicatory{Dedicated to professor Oswaldo Lezama for his brilliant academic career at the \\ Universidad Nacional de Colombia - Sede Bogot\'a}

\begin{abstract}
We establish necessary or sufficient conditions to guarantee that skew Poincar\'e-Birkhoff-Witt extensions are NI or NJ rings. Our results extend those corresponding for skew polynomial rings and establish similar properties for other families of noncommutative rings such as universal enveloping algebras and examples of differential operators.

\end{abstract}

\maketitle


\section{Introduction}\label{introduction}


\label{intro}

Several kinds of rings are defined in terms of their set of nilpotent elements. For example, a ring $R$ is called {\em NI} if its set $N(R)$ of nilpotent elements coincides with its upper radical $N^{*}(R)$
(the sum of all its nil ideals of $R$). If $N(R) = J(R)$, where $J(R)$ is the Jacobson radical of $R$, then $R$ is called {\em NJ}. $R$ is said to be {\em 2-primal} if $N(R)$ is equal to the prime radical $N_{*}(R)$ of $R$ (the intersection of all prime ideals of $R$). $R$ is called {\em weakly 2-primal} if $N(R)$ coincides with its Levitzki
radical $L(R)$ (the sum of all locally nilpotent ideals of $R$). 

The NI and NJ rings have recently been investigated by several authors. For instance, Hwang et al. \cite{Hwang2006} studied the structure of NI rings
related to strongly prime ideals and showed that minimal strongly
prime ideals can be lifted in NI rings (\cite{Hwang2006}, Theorem 2.3). They proved that for an NI ring $R$, $R$ is {\em weakly
pm} (every strongly prime ideal of $R$ is contained in a unique maximal ideal of $R$) if and only if the topological space of maximal ideals of $R$ is a retract of the topological space of strongly prime ideals of $R$, or equivalently, if the topological space of strongly prime ideals of $R$ is normal (\cite{Hwang2006}, Theorem 3.7). Also, they proved that $R$ is weakly pm if and only if $R$ is {\em pm} (every prime ideal of $R$ is contained in a unique maximal ideal of $R$) when $R$ is a symmetric ring (that is, $rst =
0$ implies $rts = 0$, where $r, s, t\in
 R$) (\cite{Hwang2006}, Theorem
3.8).

Concerning skew polynomial rings (also known as Ore extensions) introduced  by Ore \cite{Ore1933}, Bergen and Grzeszczuk \cite{Bergen2012} studied the Jacobson radical of skew polynomial rings of
derivation type $R[x; \delta]$ when the base ring $R$ has no nonzero
nil ideals. They proved that if $R$ is an algebra with no
nonzero nil ideals satisfying the acc condition on right annihilators of
powers, then $J(R[x;\delta])=0$ (\cite{Bergen2012}, Theorem 2). In the case that $R$ is a semiprime algebra where every nonzero ideal contains a
normalizing element, then $J(R[x;\delta])=0$ (\cite{Bergen2012}, Theorem 3). Related to this topic, Nasr-Isfahani
\cite{Nasr-Isf2014} gave necessary and sufficient conditions for a
skew polynomial ring of derivation type $R[x; \delta]$ to be
semiprimitive when $R$ has no nonzero nil ideals (\cite{Nasr-Isf2014}, Corollary 2.2). He also proved that $J(R[x; \delta]) = N(R[x;
\delta]) = N(R)[x; \delta]$ if and only if $N(R)$ is a $\delta$-ideal of $R$ (i.e., $N(R)$ is an ideal of $R$ and $\delta(N(R))\subseteq N(R)$) and  $N(R[x; \delta]) = N(R)[x;
\delta]$ (\cite{Nasr-Isf2014}, Proposition 2.7). Now, according to
\cite{Nasr-Isf2014}, Proposition 2.8, if $R[x; \delta]$ is NI then $J(R[x; \delta]) = N(R[x; \delta])=N(R)[x; \delta]=N^{*}(R)[x;
\delta]= N^{*}(R[x; \delta])$. 

Later, Nasr-Isfahani \cite{Nasr-Isf2015} computed the Jacobson radical of an NI $\mathbb{Z}$-graded ring $R = \bigoplus_{i\in \mathbb{Z}}R_i$. He showed that $J(R) = N(R)$ if and only if $R$ is NI ring and $J(R)\cap R_0$ is nil (\cite{Nasr-Isf2015}, Theorem 2.4). He also proved that $R[x; \sigma]$ is NJ if and only if $R[x; \sigma]$ is NI and $J(R[x; \sigma])\cap R \subseteq N(R)$ (\cite{Nasr-Isf2015}, Corollary 2.5 (1)). For a skew polynomial ring of mixed type $R[x; \sigma, \delta]$, he showed that $R[x; \sigma, \delta]$ is NI
and $N(R)$ is $\sigma$-rigid (i.e,  $r\sigma(r)\in N(R)$ implies
$r\in N(R)$, where $r\in R$) if and only if $N(R)$ is a
$\sigma$-invariant ideal  of $R$ ($N(R)$ is an ideal and
$\sigma^{-1}(N(R))=N(R)$) and $N(R[x; \sigma, \delta])=N(R)[x; \sigma, \delta]$, and equivalently, $N(R)$ is a $\sigma$-rigid ideal of $R$ and $N^{*}(R[x; \sigma,
\delta])=N^{*}(R)[x; \sigma, \delta]$ (\cite{Nasr-Isf2015}, Theorem 3.1). 


Jiang et al., \cite{Jiang2019} studied the relationship between NJ rings and some families of rings. They investigated extensions as Dorroh, Nagata, and Jordan. For a ring $R$ and an automorphism $\sigma$ of $R$, they proved that if $R$ is weakly 2-primal $\sigma$-compatible (following Annin \cite{Annin2004}, Definition 2.1, $R$ is said to be $\sigma$-{\em compatible} if for each $a, b\in R$, $ab = 0$ if and only if $a\sigma(b)=0$), then $R[x; \sigma]$ is NJ (\cite{Jiang2019}, Theorem 3.10 (1)), and if $R$ is a weakly 2-primal $\delta$-compatible ring ($R$ is said to be $\delta$-{\em compatible} if for each $a, b\in R$, $ab = 0$ implies $a\delta(b)=0$ (\cite{Annin2004}, Definition 2.1)), then $R[x; \delta]$ is NJ (\cite{Jiang2019}, Theorem 3.12 (1)). Moreover, they considered some topological conditions for NJ rings and showed relations between algebraic and topological notions (\cite{Jiang2019}, Section 4). 

Han et al., \cite{Han2021} showed that if $R$ is an NI ring, $a, b\in R$, and $ab\in Z(R)$, the center of $R$, then $ab-ba\in N^{*}(R)$, and there exists $l \geq 1$ such that $(ab)^n= (ba)^n$, for every $n\geq l$ (\cite{Han2021}, Theorem 1.3 (3)). They also proved that for the
ideal $I$ of $R$ generated by the subset $\{ab- ba\mid a, b\in R,\ {\rm such\ that}\ ab\in Z(R)\}$, if $R$ is NI then $I$ is nil and $R/I$ is an Abelian (i.e., the idempotents of $R$ are central elements of $R$) NI ring (\cite{Han2021}, Theorem 1.3 (5)).


With respect to the objects of interest in this paper, the {\em skew Poincar\'e-Birkhoff-Witt} ({\em PBW}) {\em extensions}, these were defined by Gallego and Lezama \cite{LezamaGallego} with the aim of generalizing families of noncommutative rings as PBW extensions introduced by Bell and Goodearl \cite{BellGoodearl1988}, skew polynomial rings (of injective type) defined by Ore \cite{Ore1933}, and others as solvable polynomial rings, diffusion algebras, some types of Auslander-Gorenstein rings, some Calabi-Yau and skew Calabi-Yau algebras, some Artin-Schelter regular algebras, some Koszul algebras, and others (see \cite{Fajardoetal2020} or \cite{SuarezLezamaReyes2107-1} for a detailed reference to each of these families). Several ring and theoretical properties of skew PBW extensions have been studied by different authors (see Artamonov \cite{Artamonov2015}, Fajardo et al., \cite{Fajardoetal2020}, Hamidizadeh et al., \cite{Hamidizadehetal2020}, Hashemi et al., \cite{HashemiKhalilAlhevaz2017},  \cite{HashemiKhalilGhadiri2019}, Lezama et al., \cite{JimenezLezama2016}, \cite{Lezama2020},  \cite{LezamaGallego2016}, \cite{LezamaGomez2019}, \cite{LezamaVanegas2020b}, Tumwesigye et al., \cite{TumwesigyeRichterSilvestrov2019}, and Zambrano \cite{Zambrano2020}). In particular, the 2-primal property for these objects has been investigated by Hashemi et al., \cite{HashemiKhalilAlhevaz2019}, Louzari and Reyes \cite{LouzariReyes2020}, and the authors \cite{ReyesSuarez2019Radicals}. Nevertheless, NI and NJ properties have not been studied for skew PBW extensions, so it is a natural task to find necessary or sufficient conditions under which they can be NI and NJ. This is the objective of the paper, and therefore our results contribute to the study of ideals and radicals of skew PBW extensions that has been partially carried out (see \cite{HashemiKhalilAlhevaz2019},  \cite{AcostaLezamaReyes},  \cite{LouzariReyes2020}, \cite{NinoReyesRamirez2020}, \cite{NinoReyes2020}, and \cite{ReyesSuarez2018-2}), and establish ring-theoretical properties for noncommutative rings not considered in the literature. As a matter of fact, we generalize some results appearing in Jiang et al.,  \cite{Jiang2019}, and Nasr-Isfahani \cite{Nasr-Isf2014} and \cite{Nasr-Isf2015}.

The paper is organized as follows. In Section \ref{section-prelim}, we recall definitions and properties needed for the rest of the paper. Section \ref{sect-NIskew} contains the results about the NI property for skew PBW extensions. More exactly, for $A=\sigma(R)\langle x_1,\dots, x_n\rangle$ a skew PBW extension over a ring $R$, we prove the following results: if $R$ is weak $(\Sigma, \Delta)$-compatible, then $R$ is NI if and only if $A$ is NI (Theorem \ref{teo- RweakCompNIiff A NI}). If $A$ is of derivation type over $R$, then  $A$ is NI if and
only if $N(R)$ is a $\Delta$-invariant ideal of $R$ and
$N(A)=N(R)\langle x_1,\dots, x_n\rangle$ (Proposition
\ref{prop-derivtype}). $A$ is NI and $N(R)$ is $\Sigma$-rigid if and only if $N(R)$ is a $\Sigma$-ideal of $R$ and $N(A) =
N(R)\langle x_1, \dots, x_n \rangle$, and  equivalently,  $N(R)$ is a
$\Sigma$-rigid ideal of $R$ and $N^{*}(A)= N^{*}(R)\langle x_1,
\dots, x_n \rangle$ (Theorem \ref{teo.generteo3.1}). Next, in Section \ref{sect-NJ},  results about NJ
property and their relations with NI property for skew PBW extensions are presented. We prove that if $A$ is a
graded skew PBW extension over an $\mathbb{N}$-graded ring $R=\bigoplus_{n\in \mathbb{N}}R_n$, then
$A$ is NJ if and only if $A$ is NI and $J(A)\cap R_0$ is a nil ideal
(Theorem \ref{teo-NJgradedSkew}). If furthermore $A$ is connected, then  $A$ is NJ if and only if $A$ is NI (Corollary
\ref{cor-ConecNJiifNI}). If $A$ is quasi-commutative, bijective and $R$ is a weakly 2-primal weak $\Sigma$-compatible ring, then $A$ is NJ (Theorem
\ref{teo-genteo3.10}). If $A$ is of derivation type, then
$A$ is NI if and only if $A$ is NJ, and equivalently, $R$ is NI and $N(A) = N(R)\langle x_1,\dotsc, x_n\rangle$ if and only if $R$ is NI and $N^{*}(A) = N^{*}(R)\langle x_1,\dotsc, x_n\rangle$ (Proposition \ref{cor.gencor3.2}). Finally, if $A$ is quasi-commutative, then $A$
is NJ and $N(A)=N(R)\langle x_1, \dots, x_n \rangle$ if and only if $N(R)$ is a $\Sigma$-ideal of $R$ and
$N(A)=N(R)\langle x_1, \dots, x_n \rangle$, or equivalently, $A$ is NI and $N(R)$ is $\Sigma$-rigid if and only if $N(R)$ is
$\Sigma$-rigid ideal of $R$ and $N^{*}(A)= N^{*}(R)\langle x_1,
\dots, x_n \rangle$ (Proposition \ref{prop-quasicomm}).

Finally, we present some ideas for a future work.

\section{Preliminaries}\label{section-prelim}

Throughout the paper, the term ring means an associative ring with
identity not necessarily commutative. For a ring $R$, as we saw above, we fix the following notation: $N_{*}(R)$ is its {\em prime radical}, $N^{*}(R)$ is its {\em upper radical}, 
$N(R)$ is its {\em set of nilpotent elements}, $J(R)$ is its {\em Jacobson radical}, and $L(R)$ is its {\em Levitzki radical}. It is well-known that the following relations hold: $N_{*}(R)\subseteq L(R)\subseteq N^{*}(R)\subseteq
N(R)$ and $N_{*}(R)\subseteq L(R)\subseteq N^{*}(R)\subseteq J(R)$ (see \cite{Lam1991}, \cite{Marks2003} or \cite{Nasr-Isf2014}, for more details). 

$R$ is called \emph{nil-semisimple} if it has no nonzero nil ideals. Recall that nil-semisimple rings are semiprime (\cite{Hwang2006}, p. 187). If $P$ is a prime ideal of $R$, $P$ is called \emph{completely prime} if $R/P$ is a domain, and $P$ is said to be \emph{strongly prime} if $R/P$
is nil-semisimple. Note that maximal ideals and completely prime ideals are strongly prime; any
strongly prime ideal contains a minimal strongly prime ideal;  $N^{*}(R)$ is the unique maximal nil ideal of $R$; $N^{*}(R) = \{a\in R\mid RaR \text{ is a nil ideal of } R\} = \bigcap \{P \mid P \text{
is a  strongly prime ideal of } R\}$ (see \cite{Hwang2006} or \cite{Lam1991} for more details). 

For a ring $R$, 2-primal implies
weakly 2-primal, whence NI and NJ are examples of NI rings. $R$ is
reduced (i.e., without nonzero nilpotent elements) if and only if $R$ is nil-semisimple and NI, or equivalently, $R$ is semiprime and 2-primal (\cite{Hwang2006}, p. 187). As we can check, $R$ is 2-primal if and only if $N_{*}(R)= N^{*}(R)=
N(R)$. Shin \cite{Shin1973}, Proposition 1.11, proved that the set of nilpotent elements of a ring $R$ coincides with the prime radical
if and only if every minimal prime ideal of $R$ is completely prime. Thus, $R$ is 2-primal if and only if every minimal prime ideal of
$R$ is completely prime, or equivalently, $R/N_{*}(R)$ is reduced. 

A ring $R$ is said to be \emph{semicommutative} if for every pair of
elements $a, b\in R$, we have that $ab = 0$ implies $aRb = 0$. Shin \cite{Shin1973}, Lemma 1.2 and Theorem 1.5, established that semicommutative rings are 2-primal, and hence semi-commutative rings are NI. Note that domains are reduced rings, reduced rings are symmetric, symmetric rings are reversible ($R$ is said to be {\em reversible} if $ab = 0$ implies $ba = 0$, where $a, b\in R$), and reversible rings are semicommutative, but the converses are not true in general (see \cite{Marks2003}). A ring is called
\emph{right} (left) {\em duo} if every right (left) ideal of it is
two-sided. Shin \cite{Shin1973}, Lemma 1.2, showed that right (left) duo rings are semicommutative. Therefore, NI rings contain several families of rings such as domains, reduced rings, symmetric rings, semi-commutative rings, reversible rings, one-sided duo rings, 2-primal rings and NJ rings (for more details, see \cite{Marks2001}). K\"othe's
conjecture (the upper nilradical contains every nil left
ideal) holds for NI rings (Hwang et al., \cite{Hwang2006}, p. 192).

Equivalent definitions for NI rings are presented in Proposition \ref{prop-DefeqNI}.

 \begin{proposition}[\cite{Hwang2006}, Lemma 2.1]\label{prop-DefeqNI} For a ring $R$, the following conditions are equivalent.
 \begin{enumerate}
 \item[\rm (i)] $R$ is NI.
 \item[\rm (ii)] $N(R)$ is an ideal.
 \item[\rm (iii)] Every subring {\rm (}possibly without identity{\rm )} of $R$ is NI.
\item[\rm (iv)] Every minimal strongly
prime ideal of $R$ is completely prime.
\item[\rm (v)] $R/N^{*}(R)$ is a reduced ring.
\item[\rm (vi)] $R/N^{*}(R)$ is a symmetric ring.
\end{enumerate}
\end{proposition}

Some examples of NJ rings are nil rings, division rings, Boolean
rings, commutative Jacobson rings, commutative affine algebras over a field $\mathbb{K}$, semi-Abelian $\pi$-regular rings, locally finite Abelian rings (Jiang et al., \cite{Jiang2019}, Example 2.5). Every reduced ring is NJ (\cite{Jiang2019}, Proposition
2.11). Note that $R = \mathbb{Z}[[x]]$ is a domain and
hence NI, with $N(R) = \{0\}$, but $J(\mathbb{Z}[[x]]) =
x\mathbb{Z}[[x]]\neq \{0\}$ and so $R$ is not
an NJ ring (\cite{Jiang2019}, Example 2.2). This example shows that NI are not included in NJ rings.

Next, we present the objects of interest in this paper with some of their ring-theoretical notions. The symbol $\mathbb{N}$ denotes the set of natural numbers including the zero element.

\begin{definition}[\cite{LezamaGallego}, Definition 1]\label{def.skewpbwextensions}
Let $R$ and $A$ be rings. $A$ is called a \textit{skew PBW extension
over} $R$, denoted by $A=\sigma(R)\langle x_1,\dots,x_n\rangle$, if the
following conditions hold:
\begin{enumerate}
\item[\rm (i)]$R$ is a subring of $A$ sharing the same identity element.
\item[\rm (ii)] there exist finitely many elements $x_1,\dots ,x_n\in A$ such that $A$ is a left free $R$-module,
    with basis the
set of standard monomials
\begin{center}
${\rm Mon}(A):= \{x^{\alpha}:=x_1^{\alpha_1}\cdots
x_n^{\alpha_n}\mid \alpha=(\alpha_1,\dots ,\alpha_n)\in
\mathbb{N}^n\}$.
\end{center}
Moreover, $x^0_1\cdots x^0_n := 1 \in {\rm Mon}(A)$.
\item[\rm (iii)]For each $1\leq i\leq n$ and any $r\in R\ \backslash\ \{0\}$, there exists an
    element $c_{i,r}\in R\ \backslash\ \{0\}$ such that $x_ir-c_{i,r}x_i\in R$.
\item[\rm (iv)]For $1\leq i,j\leq n$, there exists $d_{i,j}\in R\ \backslash\ \{0\}$ such that
\begin{equation}\label{sigmadefinicion2}
x_jx_i-d_{i,j}x_ix_j\in R+Rx_1+\cdots +Rx_n.
\end{equation}
\end{enumerate}

Note that every element $f\in A\ \setminus\ \{0\}$ can be expressed
uniquely as $f= a_0 + a_1X_1+\dotsb +a_mX_m$, with $a_i\in R\setminus
\{0\}$ and $X_i\in {\rm Mon}(A)$, $0 \leq i\leq m$ (\cite{LezamaGallego}, Remark 2). For $X =
x^{\alpha}=x_1^{\alpha_1}\cdots x_n^{\alpha_n}\in {\rm Mon}(A)$,
$\deg(X)=|\alpha|:= {\alpha_1}+\cdots + {\alpha_n}$. 
\end{definition}

\begin{proposition}[\cite{LezamaGallego}, Proposition 3]\label{sigmadefinition}
Let $A=\sigma(R)\langle x_1,\dots,x_n\rangle$ be a skew PBW
extension over $R$. For each $1\leq i\leq n$, there exist an
injective endomorphism $\sigma_i:R\rightarrow R$ and a
$\sigma_i$-derivation $\delta_i:R\rightarrow R$ such that
$x_ir=\sigma_i(r)x_i+\delta_i(r)$, for every $r \in R$.
\end{proposition}

From now on $\Sigma:= \{\sigma_1,\dotsc, \sigma_n\}$ and $\Delta:=
\{\delta_1,\dotsc, \delta_n\}$. We say that $(\Sigma, \Delta)$ is a
\emph{system of endomorphisms and $\Sigma$-derivations} of $R$ with
respect to $A$. For $\alpha=(\alpha_1,\dots,\alpha_n)\in
\mathbb{N}^n$, $\sigma^{\alpha}:=\sigma_1^{\alpha_1}\circ \cdots
\circ \sigma_n^{\alpha_n}$, $\delta^{\alpha}:=
\delta_1^{\alpha_1}\circ \cdots \circ \delta_n^{\alpha_n}$, where
$\circ$ denotes composition of functions.

Following \cite{LezamaGallego}, Definition 4,  skew PBW extension $A$ is called \textit{bijective} if $\sigma_i$
is bijective  and $d_{i,j}$ is invertible for any $1\leq i<j\leq n$.
 $A$ is called \textit{quasi-commutative} if the conditions
{\rm(}iii{\rm)} and {\rm(}iv{\rm)} in Definition
\ref{def.skewpbwextensions} are replaced by the following:
\begin{enumerate}
\item[\rm (iii')] for each $1\leq i\leq n$ and all $r\in R\ \backslash\ \{0\}$, there exists $c_{i,r}\in R\ \backslash\ \{0\}$ such that
$x_ir=c_{i,r}x_i$;
\item[\rm (iv')]for any $1\leq i,j\leq n$, there exists $d_{i,j}\in R\ \backslash\ \{0\}$ such that $x_jx_i=d_{i,j}x_ix_j$.
\end{enumerate}
If $\sigma_i$ is the identity map of $R$, for each $i =1,\dotsc, n$, (we write $\sigma_i={\rm id}_R$), we say that
$A$ is a skew PBW extension of \textit{derivation type}. Similarly, if $\delta_i=0$, for each $\delta_i\in \Delta$, then $A$ is said to be a skew PBW extension of \textit{endomorphism type}.
\begin{remark}\label{comparison}
\begin{enumerate}
\item [\rm (i)] From Definition \ref{def.skewpbwextensions} (iv), it is clear that skew PBW extensions are more general than iterated skew polynomial rings (cf. \cite{GomezSuarez2019}). For example, universal enveloping algebras of finite dimensional Lie algebras and some 3-dimensional skew polynomial algebras in the sense of Bell and Smith \cite{BellSmith1990} (see also Rosenberg \cite{Rosenberg1995}) cannot be expressed as iterated skew polynomial rings but are skew PBW extensions. For quasi-commutative skew PBW extensions, these are isomorphic to iterated Ore extensions of endomorphism type (\cite{LezamaReyes2014}, Theorem 2.3). 
\item [\rm (ii)] Skew PBW extensions of endomorphism type are more general than iterated Ore extensions of endomorphism type. Let us illustrate the situation with two and three indeterminates.
		
	For the iterated Ore extension of endomorphism type $R[x;\sigma_x][y;\sigma_y]$, if $r\in R$ then we have the following relations: $xr = \sigma_x(r)x$, $yr = \sigma_y(r)y$, and $yx = \sigma_y(x)y$. Now, if we have $\sigma(R)\langle x, y\rangle$ a skew PBW extension of endomorphism type over $R$, then for any $r\in R$, Definition \ref{def.skewpbwextensions} establishes that $xr=\sigma_1(r)x$, $yr=\sigma_2(r)y$, and $yx = d_{1,2}xy + r_0 + r_1x + r_2y$, for some elements $d_{1,2}, r_0, r_1$ and $r_2$ belong to $R$. From these relations it is clear which one of them is more general.
	
If we have the iterated Ore extension $R[x;\sigma_x][y;\sigma_y][z;\sigma_z]$, then for any $r\in R$, $xr = \sigma_x(r)x$, $yr = \sigma_y(r)y$, $zr = \sigma_z(r)z$, $yx = \sigma_y(x)y$, $zx = \sigma_z(x)z$, $zy = \sigma_z(y)z$. For the skew PBW extension of endomorphism type $\sigma(R)\langle x, y, z\rangle$, $xr=\sigma_1(r)x$, $yr=\sigma_2(r)y$, $zr = \sigma_3(r)z$, $yx = d_{1,2}xy + r_0 + r_1x + r_2y + r_3z$, $zx = d_{1,3}xz + r_0' + r_1'x + r_2'y + r_3'z$, and $zy = d_{2,3}yz + r_0'' + r_1''x + r_2''y + r_3''z$, for some elements $d_{1,2}, d_{1,3}, d_{2,3}, r_0, r_0', r_0'', r_1, r_1', r_1'', r_2, r_2', r_2'', r_3$, $r_3', r_3''$ of $R$. As the number of indeterminates increases, the differences between both algebraic structures are more remarkable.
\item [\rm (iii)] PBW extensions introduced by Bell and Goodearl \cite{BellGoodearl1988} are particular examples of skew PBW extensions (see \cite{LezamaGallego}, Section 1 for a detailed description). By definition, the first objects satisfy the relation $x_ir = rx_i + \delta_i(r)$, so that these structures are examples of skew PBW extensions of derivation type. As examples, we mention the following:  the enveloping algebra of a finite-dimensional Lie algebra; any differential operator ring $R[\theta_1,\dotsc, \theta_1;\delta_1,\dotsc, \delta_n]$ formed from commuting derivations $\delta_1,\dotsc, \delta_n$; differential operators introduced by Rinehart; twisted or smash product differential operator rings, and others (for more details, see \cite{BellGoodearl1988}, p. 27). 
\end{enumerate}
\end{remark}

From Definition \ref{def.skewpbwextensions}, it follows that skew PBW extensions are not $\mathbb{N}$-graded ring in a non-trivial sense. With this in mind, Proposition \ref{prop.grad A} allows to define a subfamily of these extensions, the {\em graded skew PBW extensions} (Definition \ref{def. graded skew PBW ext}) that were introduced by the first author in \cite{Suarez}. Before presenting the definition, we recall that if $R=\bigoplus_{p\in \mathbb{N}}R_p$ and $S=\bigoplus_{p\in \mathbb{N}}S_p$ are $\mathbb{N}$-graded rings, then a map $\varphi : R\to S$ is called \emph{graded} if $\varphi(R_p)\subseteq
S_p$, for each $p\in \mathbb{N}$. For $m\in \mathbb{N}$,
$R(m):=\bigoplus_{p\in \mathbb{N}}R(m)_p$, where $R(m)_p:=R_{p+m}$.

\begin{proposition}[\cite{Suarez}, Proposition
2.7(ii)]\label{prop.grad A} Let $A=\sigma(R)\langle x_1,\dots,
x_n\rangle$ be a bijective skew PBW extension over an
$\mathbb{N}$-graded algebra $R=\bigoplus_{m\geq 0}R_m$. If the
following conditions hold:
\begin{enumerate}
\item[\rm (i)] $\sigma_i$ is a graded ring homomorphism and $\delta_i : R(-1) \to R$ is a graded $\sigma_i$-derivation, for all $1\leq i  \leq n$, and
\item[\rm (ii)]  $x_jx_i-d_{i,j}x_ix_j\in R_2+R_1x_1 +\cdots + R_1x_n$, as in {\rm (}\ref{sigmadefinicion2}{\rm )} and $d_{i,j}\in R_0$,
\end{enumerate}
then $A$ is an $\mathbb{N}$-graded algebra with graduation given by
$A = \bigoplus_{p\geq 0} A_p$, where for $p\geq 0$, $A_p$ is the
$\mathbb{K}$-space generated by the set
\[\Bigl\{r_tx^{\alpha} \mid t+|\alpha|= p,\  r_t\in R_t \text{  and } x^{\alpha}\in {\rm
Mon}(A)\Bigr\}.
\]
\end{proposition}

\begin{definition}[\cite{Suarez}, Definition 2.6]\label{def. graded skew PBW ext}
Let $A=\sigma(R)\langle x_1,\dots, x_n\rangle$ be a bijective skew
PBW extension over an $\mathbb{N}$-graded algebra
$R=\bigoplus_{m\geq 0}R_m$. If $A$ satisfies the conditions (i) and
(ii) established in Proposition \ref{prop.grad A}, then we say that
$A$ is a {\em graded skew PBW extension over} $R$.
\end{definition}

We present some remarkable examples of graded skew PBW extensions.

\begin{example}\label{ejem-clifford}
The Jordan plane, homogenized enveloping algebras, and some classes of diffusion algebras (\cite{Suarez}, Examples 2.9). Graded Clifford algebras defined by Le Bruyn
\cite{Le Bruyn 1995} are also examples of graded skew PBW extensions if we assume the condition of PBW basis. Let us see the details.

Following Cassidy and Vancliff \cite{Cassidy2010}, let $\mathbb{K}$ be an algebraically
closed field such that ${\rm char}(\mathbb{K})\neq 2$ and let $M_1,\dots,M_n\in \mathbb{M}_n(\mathbb{K})$ be symmetric matrices of order $n\times n$ with entries in $\mathbb{K}$. A \emph{graded Clifford
algebra} $\mathcal{A}$ is a $\mathbb{K}$-algebra  on degree-one generators $x_1,\dots, x_n$ and on degree-two generators $y_1,\dots, y_n$ with defining relations given by:
\begin{enumerate}
\item[(i)] $x_ix_j+x_jx_i=\sum_{k=1}^n(M_k)_{ij}y_k$ for all
$i,j=1,\dots,n$;
\item[(ii)] $y_k$ central for all $k = 1,\dots,n$.
\end{enumerate}
Note that the commutative polynomial ring $R=\mathbb{K}[y_1,\dots,
y_n]$ is an $\mathbb{N}$-graded algebra where $R_0=\mathbb{K}, R_1= \{0\}$, 
$y_1,\dots, y_n\in R_2$, and $R_i = \{0\}$, for $i\ge 3$. If we suppose that the set $\{x_1^{a_1}\dotsb x_n^{a_n}\mid a_i\in \mathbb{N}, i=1,\dotsc, n\}$ is a left PBW $R$-basis for $\mathcal{A}$, then the graded Clifford algebra $\mathcal{A}$ is a graded skew PBW extension over the connected algebra $R$, that is, $\mathcal{A} \cong \sigma(R)\langle x_1,\dotsc, x_n\rangle$. Indeed, from the relations (i) and (ii) above, it is clear that 
$\sigma_i = {\rm id}_R$, $\delta_i=0$, $d_{i,j}=-1\in
R_0$, for $1\leq i,j\leq n$,  and $\sum_{k=1}^n(M_k)_{ij}y_k\in R_2$, where $d_{i,j}$ is given as in expression (\ref{sigmadefinicion2}). In this way, $\mathcal{A}$ is a bijective skew PBW extension that satisfies the conditions of
Proposition \ref{prop.grad A}. 
\end{example}

Let $A=\sigma(R)\langle x_1, \dots, x_n\rangle$ be a skew PBW extension over a ring $R$, and consider the sets $\Sigma= \{\sigma_1, \dots, \sigma_n\}$ and
$\Delta:= \{\delta_1, \dots, \delta_n\}$ as above. An ideal $I$ of $R$ is
called $\Sigma$-\emph{ideal} if $\sigma_i(I)=I$, for each $1\leq
i\leq n$. From Lezama et al., \cite{AcostaLezamaReyes}, Definition 2.1, $I$ is called $\Sigma$-\emph{invariant} if
$\sigma_i(I)\subseteq I$, and it is called $\Delta$-\emph{invariant}
if $\delta_i(I)\subseteq I$, for $1\leq i\leq n$. If $I$ is both $\Sigma$ and $\Delta$-invariant, we
say that $I$ is $(\Sigma,\Delta)$-\emph{invariant}. From Reyes 
\cite{Reyes2015}, Definition 3.2, $R$ is $\Sigma$-\emph{rigid} if
    $r\sigma^{\alpha}(r)=0$ implies $r=0$, where $r\in R$ and
    $\alpha\in \mathbb{N}^n$. A subset $S\subseteq
R$ is $\Sigma$-\emph{rigid} if
    $r\sigma^{\alpha}(r)\in S$ implies $r\in S$, where $r\in R$ and
    $\alpha\in \mathbb{N}^n$. Following Hashemi et al., \cite{HashemiKhalilAlhevaz2019}, for $S\subseteq R$, we denote the set of all elements of $A$ with
coefficients in $S$ by $S\langle x_1,\dots,x_n\rangle$. Definitions of $(\Sigma,\Delta)$-ideal $I$ and $I\langle
x_1,\dots,x_n\rangle=I\langle x_1,\dots,x_n; \Sigma, \Delta\rangle$ introduced by Hashemi et al., \cite{HashemiKhalilAlhevaz2019} are the same as those of $(\Sigma,\Delta)$-invariant ideal $I$ and $IA$, 
respectively, considered in \cite{AcostaLezamaReyes}. Note that the terminology
used by Nasr-Isfahani \cite{Nasr-Isf2015} is different, since for $\sigma$ an
endomorphism of $R$, $\delta$ a $\sigma$-derivation of $R$, and $I$ an ideal of $R$, $I$ is called $\sigma$-{\em invariant} if
$\sigma^{-1}(I) = I$, and it is called a $\delta$-{\em ideal} if
$\delta(I)\subseteq I$; this same terminology is used in Nasr-Isfahani
\cite{Nasr-Isf2014}, \cite{Nasr-Isf2015}. Finally, from \cite{AcostaLezamaReyes},
Proposition 2.6 (i), if $I$ is a $(\Sigma,\Delta)$-invariant ideal
of $R$, then $IA=I\langle x_1,\dots,x_n; \Sigma,
\Delta\rangle=I\langle x_1,\dots,x_n\rangle $ is an ideal of $A$,
$IA\cap R = I$ and $IA$ is proper if and only if $I$ is proper.

Next, we present different ring-theoretical notions that have been defined for skew PBW extensions and extend corresponding notions for skew polynomial rings (for more details, see  \cite{ReyesSuarez2018-3}).
\begin{definition}
	Let $A = \sigma(R)\langle x_1,\dotsc, x_n\rangle$ be a skew PBW extension over $R$.
\begin{enumerate}
	\item [\rm (i)] (\cite{ReyesSuarez2018-3}, Definition 3.3) $R$ is called a  $\Sigma$-\emph{skew Armendariz} ring if for elements $f=\sum_{i=0}^{m} a_iX_i$ and $g = \sum_{j=0}^{t} b_jY_j$ in $A$, the equality $fg=0$ implies $a_i\sigma^{\alpha_i}(b_j)=0$,  for all $0\le i\le m$ and $0\le j\le t$, where $\alpha_i={\rm exp}(X_i)$.   $R$ is called a \emph{weak $\Sigma$-skew Armendariz} ring if for elements $f=\sum_{i=0}^{n} a_ix_i$ and $g = \sum_{j=0}^{n} b_jx_j$ in $A$ $(x_0:=1)$, the equality $fg=0$ implies $a_i\sigma_i(b_j)=0$, for all $0\le i, j \le n$ $(\sigma_0:={\rm id}_R)$.
    \item [\rm (ii)] (\cite{HashemiKhalilAlhevaz2017}, Definition 3.2; \cite{ReyesSuarez2018-2}, Definition 3.1) $R$ is said to be $\Sigma$-{\em compatible} if for each $a,b\in R$,
        $a\sigma^{\alpha}(b) = 0$ if and only if $ab=0$, where $\alpha\in \mathbb{N}^{n}$.
        $R$ is said to be $\Delta$-{\em compatible} if
        for each $a,b \in R$,  $ab=0$ implies $a\delta^{\beta}(b)=0$, where $\beta \in \mathbb{N}^{n}$.
        If $R$ is both
        $\Sigma$-compatible and $\Delta$-compatible, then $R$ is called
        $(\Sigma, \Delta)$-{\em compatible}.
        \item [\rm (iii)] (\cite{ReyesSuarez2019-2}, Definition 4.1) $R$ is said to be {\em weak $\Sigma$-compatible} if for each $a,b\in R$,
        $a\sigma^{\alpha}(b) \in N(R)$ if and only if $ab \in N(R)$, where $\alpha\in \mathbb{N}^{n}$.
        $R$ is said to be {\em weak $\Delta$-compatible} if
        for each $a,b \in R$,  $ab\in N(R)$ implies $a\delta^{\beta}(b)\in N(R)$, where $\beta \in \mathbb{N}^{n}$.
        If $R$ is both
        weak $\Sigma$-compatible and  weak $\Delta$-compatible, $R$ is called
        {\em weak $(\Sigma, \Delta)$- compatible}.
        \end{enumerate}
\end{definition}

Several relations between above skew Armendariz notions and different examples can be found in \cite{ReyesRodriguez2019}, Example 3.4, and \cite{ReyesSuarez2019-1}, Section 5.

\begin{remark}
Let $A=\sigma(R)\langle x_1, \dots, x_n\rangle$ be a skew PBW
extension over a ring $R$ and $I$ an ideal of $R$. 
\begin{enumerate}
\item[(i)] $I$ is $\Sigma$-invariant if and only if
$\sigma^{\alpha}(I)\subseteq I$, for every $\alpha\in
\mathbb{N}^{n}$.
\item[(ii)] $I$ is $\Delta$-invariant if and only if
$\delta^{\alpha}(I)\subseteq I$, for every $\alpha\in
\mathbb{N}^{n}$.
\item[(iii)] $I$ is  $\Sigma$-ideal if and only if $\sigma_i^{-1}(I)=I$, for each $1\leq i\leq n$. Also, if for each
$1\leq i\leq n$,  $\sigma_i^{-1}(I)=I$ then $\sigma^{-\alpha}(I)=I$,
for $\alpha\in \mathbb{N}^{n}$.
 Therefore, the definition of $\Sigma$-ideal given in this paper
coincides with the definition of $\Sigma$-invariant ideal given by Hashemi et al., \cite{HashemiKhalilAlhevaz2019}, Definition 3.1, and the definition
of $\alpha$-invariant ideal presented by Nasr-Isfahani \cite{Nasr-Isf2015}, p. 5116.
\end{enumerate}
\end{remark}

\section{NI skew PBW extensions}\label{sect-NIskew}
This section contains the original results of the paper about NI property for skew PBW extensions. We start with Proposition \ref{prop-skewPBWcompNI} which follows directly from \cite{LouzariReyes2020}, Theorem 3.9, and \cite{ReyesSuarez2019Radicals}, Proposition 4.4. Recall that a ring $R$ is {\em locally finite} if every finite subset in $R$ generates a finite semigroup multiplicatively.

\begin{proposition}\label{prop-skewPBWcompNI}
Let $A=\sigma(R)\langle x_1, \dots, x_n\rangle$ be a skew PBW
extension over $R$. If $R$ satisfies one of the following conditions,
\begin{enumerate}
\item[\rm (i)] $R$ is 2-primal and $(\Sigma, \Delta)$-compatible, or
\item[\rm (ii)] $R$ is locally finite, $(\Sigma,
\Delta)$-compatible and $\Sigma$-skew Armendariz,
\end{enumerate}
then $A$ is an NI ring.
\end{proposition}

\begin{proposition}\label{prop-NIskewPBWimplRNI}
If $A=\sigma(R)\langle x_1, \dots, x_n\rangle$ is an NI skew PBW
extension over $R$, then $R$ is an NI ring, $N(R)\subseteq J(R)$,
$N(A)\subseteq J(A)$, and therefore $N(R)\subseteq J(A)$.
\end{proposition}

\begin{proof}
From Definition \ref{def.skewpbwextensions} (i) we know that $R$ is a
subring of $A$. Since $A$ is NI,  Proposition
\ref{prop-DefeqNI} implies that $R$ is NI. Now, since the Jacobson
radical of a ring contains the nil ideals, by Proposition
\ref{prop-DefeqNI} we have that $N(A)$ and $N(R)$ are nil ideals of
$A$ and $R$, respectively. Hence, $N(R)\subseteq J(R)$ and
$N(A)\subseteq J(A)$. As $N(R)\subseteq N(A)$, then $N(R)\subseteq
J(A)$.
\end{proof}

From Propositions \ref{prop-skewPBWcompNI} and \ref{prop-NIskewPBWimplRNI} we deduce that if $R$ is
locally finite, $(\Sigma, \Delta)$-compatible and $\Sigma$-skew
Armendariz, then $R$ is NI. This result has been proved by Reyes and Su\'arez \cite{ReyesSuarez2019Radicals}, Theorem
4.3.

Since weak $(\Sigma, \Delta)$-compatible rings are more general than $(\Sigma, \Delta)$-compatible rings, and NI rings are more general than 2-primal rings, the following theorem generalizes \cite{HashemiKhalilAlhevaz2019}, Theorem 4.11, and some other results of \cite{HashemiKhalilAlhevaz2019} and \cite{LouzariReyes2020} formulated for skew PBW extensions over 2-primal $(\Sigma, \Delta)$-compatible rings.

\begin{theorem}\label{teo- RweakCompNIiff A NI}
If $A=\sigma(R)\langle x_1, \dots, x_n \rangle$ is a skew PBW extension over a weak $(\Sigma, \Delta)$-compatible ring $R$, then $R$ is NI if and only if $A$ is NI.
\end{theorem}

\begin{proof} Suppose that $R$ is an NI ring.
Let us first see that $N(R)$ is a $(\Sigma,\Delta)$-invariant. By Proposition \ref{prop-DefeqNI}, $N(R)$ is an ideal of
$R$. For a fixed $i$, if $\sigma_i(r)\in \sigma_i(N(R))$, where $r\in N(R)$, then $\sigma_i(r^k)=\sigma_i(r)^k=0$, for some positive integer $k$. Thus $\sigma_i(r)\in N(R)$, i.e.,
$\sigma_i(N(R))\subseteq N(R)$. Now, for $r\in N(R)$, $\delta_i(r)\in N(R)$, since $R$ is weak
$(\Sigma, \Delta)$-compatible, whence $\delta_i(N(R))\subseteq N(R)$. Since
$N(R)$ is a $(\Sigma,\Delta)$-invariant ideal or $R$, 
\cite{AcostaLezamaReyes}, Proposition 2.6 (i), implies that $N(R)\langle x_1, \dots, x_n \rangle$ is an ideal of $A$. 

Let us show that $N(R)\langle x_1, \dots, x_n \rangle=N(A)$. From \cite{ReyesSuarez2019-2}, Theorem 4.6, $f=\sum_{i=0}^ta_iX_i\in
N(R)\langle x_1, \dots, x_n \rangle$ if and only if $a_i\in N(R)$, for $0\leq i\leq t$, if and only if $f=\sum_{i=0}^ta_iX_i\in N(A)$.
Therefore, $N(A)$ is an ideal of $A$, and by Proposition
\ref{prop-DefeqNI}, we have that $A$ is an NI ring. 

Conversely, if $A$ is an NI ring, then by Proposition \ref{prop-NIskewPBWimplRNI}
we have that $R$ is an NI ring.
\end{proof}

If $A=\sigma(R)\langle x_1,\dots,x_n\rangle$ is a  skew PBW extension over a $(\Sigma,\Delta)$-compatible ring $R$, then $N(R)$ is $\Sigma$-rigid. Indeed, for $r\in R$ and $\alpha\in
\mathbb{N}^n$, if $r\sigma^{\alpha}(r)\in N(R)$, then \cite{ReyesSuarez2019-1}, Lemma 2, implies that $r^2\in N(R)$, whence $r\in N(R)$.

For the next result, recall that a ring $R$ is said to be {\em Dedekind finite} if $ab=1$ implies $ba=1$, where $a,b\in R$.

\begin{proposition}\label{prop-varios}
If $A=\sigma(R)\langle x_1,\dots,x_n\rangle$ is an NI skew PBW extension over $R$, then:
\begin{enumerate}
\item[\rm (i)] $N(R)$ and $N(A)$ are completely semiprime.
\item[\rm (ii)] $d_{i,j}$ in Definition \ref{def.skewpbwextensions} are
units. In general, left {\rm (}resp. right{\rm )} invertible elements in $A$ are
units.
\end{enumerate}
\end{proposition}
\begin{proof}

\begin{enumerate}
\item[(i)] Since $A$ and $R$ are NI rings, then $N(R)$ is an ideal
of $R$ and $N(A)$ is an ideal of $A$. If $r^2\in N(R)$ then
$(r^2)^k=r^{2k}=0$, for some positive integer $k$, and so $r\in N(R)$.
Analogously, if $f^2\in N(A)$ then $f\in N(A)$. Therefore, $R$ and
$A$ are completely semiprime.
\item[(ii)] By \cite{LezamaGallego}, Remark 2 (iii), for every $1\leq i <
j\leq n$, $d_{i,j}$ has a left inverse and $d_{j,i}$ has a right
inverse. In this way, there exist elements $r$, $s\in R$ such that
$rd_{i,j}=1=d_{j,i}s$. Since $R$ is an NI ring, then by
\cite{Hwang2006}, Proposition 2.7 (1), $R$ is Dedekind finite, whence $d_{i,j}r=1=sd_{j,i}$. Now, if $f\in A$ is left (resp. right)
invertible, then $gf=1$ (resp. $fh=1$), for some $g,h\in A$. Since
$A$ is NI, $A$ is Dedekind finite and so $fg=1$ (resp.
$hf=1$).
\end{enumerate}
\end{proof}

\begin{remark}
In \cite{ReyesRodriguez2019}, Theorem 3.14, it was presented a relation between skew PBW extensions of endomorphism type and the notion of Dedekind finite by considering a skew notion of McCoy ring.
\end{remark}

\begin{proposition}\label{prop-Deltaideal}
Let $A=\sigma(R)\langle x_1,\dots,x_n\rangle$ be a skew PBW
extension of derivation type and $I\subseteq R$. Then $I\langle
x_1,\dots,x_n\rangle$ is an ideal of $A$ if and only if $I$ is a
$\Delta$-invariant {\rm (}and therefore $(\Sigma,\Delta)$-invariant{\rm )} ideal of $R$.
\end{proposition}

\begin{proof}
If $I\langle x_1,\dots,x_n\rangle$ is an ideal of $A$, then
$I$ is an ideal of $R$. Let $\delta_i(r)\in \delta_i(I)$ such that
$r\in I$. Then $x_ir=rx_i+\delta_i(r)\in I\langle
x_1,\dots,x_n\rangle$, for each $1\leq i\leq n$. As $-rx_i\in
I\langle x_1,\dots,x_n\rangle$ then $\delta_i(r)\in I\langle
x_1,\dots,x_n\rangle$; in particular, $\delta_i(r)\in I$. This means that $I$ is
a $\Delta$-invariant ideal. Since $A$ is of derivation type, $I$ is a $\Sigma$-invariant ideal. 

The converse follows from \cite{AcostaLezamaReyes}, Proposition 2.6 (i).
\end{proof}

\begin{proposition}\label{prop-derivtype}
If $A=\sigma(R)\langle x_1,\dots,x_n\rangle$ is a skew PBW
extension of derivation type over $R$, then the following assertions hold:
\begin{enumerate}
\item[\rm (i)] $N(R)$ is $\Sigma$-rigid.
\item[\rm (ii)] If $A$ is NI then $N(R)$ and $N^{*}(R)$ are $\Sigma$-rigid ideals.
\item[\rm (iii)] For every completely prime $P$ of $A$, $P\cap R$ is a
completely prime ideal of $R$.
\item[\rm (iv)] $A$ is NI if and only if $N(R)$ is a $\Delta$-invariant ideal of $R$ and $N(A)=N(R)\langle x_1,\dots,
x_n\rangle$.
\end{enumerate}
\end{proposition}

\begin{proof}

\begin{enumerate}
\item [\rm (i)] If $r\in R$ satisfies $r^2= r\sigma^{\alpha}(r)\in N(R)$,
then $r\in N(R)$, i.e., $N(R)$ is $\Sigma$-rigid.
\item [\rm (ii)] If $A$ is NI, then $R$ is NI and therefore $N(R)$ is an ideal
of $R$. By (i), $N(R)$ is $\Sigma$-rigid, so $N(R)$ is a $\Sigma$-rigid ideal. Since $N(R)= N^{*}(R)$, then $N^{*}(R)$ is a $\Sigma$-rigid ideal.
\item [\rm (iii)] From \cite{NinoReyes2020}, Theorem 1, for every completely prime ideal $P$ of
$A$, $P\cap R$ is a completely prime ideal of $R$.
\item [\rm (iv)]If $A$ is NI, then by  Proposition \ref{prop-NIskewPBWimplRNI} $R$ is NI, and so $N(R)$ is an ideal of $R$. Let $\delta_i(r)\in \delta_i(N(R))$ such that $r\in N(R)\subseteq N(A)$.
Since $A$ is NI, $N(A)$ is an ideal of $A$ and $x_ir=rx_i+\delta_i(r)\in N(A)$, for each $1\leq i\leq n$. As
$-rx_i\in N(A)$, then $\delta_i(r)\in N(A)\cap R=N(R)$, which means that $N(R)$ is
a $\Delta$-invariant ideal. From \cite{HashemiKhalilAlhevaz2019}, Proposition 4.1, $N(A)\subseteq N(R)\langle x_1, \dots, x_n \rangle$. For the another inclusion, let $f=a_0+a_1X_1+\dots+ a_tX_t\in N(R)\langle
x_1,\dots,x_n\rangle$, with  $a_i\in N(R)\subseteq N(A)$, for $0\leq
i\leq t$. Since $N(A)$ is an ideal of $A$, $a_0, a_1X_1, \dots,
a_tX_t\in N(A)$, and so $f=a_0+a_1X_1+\dots+ a_tX_t\in N(A)$.

Conversely, if $N(R)$ is a  $\Delta$-invariant ideal, Proposition \ref{prop-Deltaideal} guarantees that $N(A) = N(R)\langle
x_1, \dots, x_n \rangle$ is an ideal of $A$. Proposition
\ref{prop-DefeqNI} implies that $A$ is an NI ring.
\end{enumerate}
\end{proof}

The following result is one of the most important in the paper. This extends Nasr-Isfahani \cite{Nasr-Isf2015}, Theorem 3.1.

\begin{theorem}\label{teo.generteo3.1}
If $A = \sigma(R)\langle x_1,\dotsc, x_n\rangle$ is a skew PBW extension over a ring $R$, then the following statements are equivalent:
\begin{enumerate}
\item[\rm (i)] $A$ is  NI and $N(R)$ is $\Sigma$-rigid.
\item[\rm (ii)] $N(R)$ is a $\Sigma$-ideal of $R$ and $N(A) =
N(R)\langle x_1, \dots, x_n \rangle$.
\item[\rm (iii)] $N(R)$ is a $\Sigma$-rigid ideal of $R$ and $N^{*}(A)= N^{*}(R)\langle x_1, \dots, x_n \rangle$.
\end{enumerate}
\end{theorem}

\begin{proof}
(i) $\Rightarrow$ (ii) Suppose that $A$ is  NI and $N(R)$
is $\Sigma$-rigid. By Proposition \ref{prop-NIskewPBWimplRNI}, $R$ is NI and so $N(R)$ is an ideal of $R$. From Proposition \ref{sigmadefinition} we know that every $\sigma_i$ is injective, whence $r\in N(R)$ if and only if $r^{k}=0$, for some positive integer $k$,
if and only if $\sigma_i(r^k)=(\sigma_i(r))^k=0$ if and only if
$\sigma_i(r)\in N(R)$, and equivalently, $r\in \sigma_i^{-1}(N(R))$, which shows that $N(R)$ is   a $\Sigma$-ideal of $R$. 

With the aim of showing that 
$N(A)=N(R)\langle x_1, \dots, x_n \rangle$, before consider the following facts.

Note that $N(R)$ is
$\Delta$-invariant; indeed, if $r\in N(R)\subseteq N(A)$, then
$x_ir=\sigma_i(r)x_i+\delta_i(r)\in N(A)$. Since $N(R)$ is a $\Sigma$-ideal, then $\sigma_i(r)\in N(R)\subseteq N(A)$, and so
$\sigma_i(r)x_i\in N(A)$, which implies that $\delta_i(r)\in N(A)$, that
is, $\delta_i(r)\in N(R)$. By \cite{AcostaLezamaReyes}, Proposition 2.2 (i), the system of endomorphisms and
$\Sigma$-derivations $(\Sigma, \Delta)$ induces over $R/N(R)$ a system $(\overline{\Sigma}, \overline{\Delta})$ of endomorphisms and $\Sigma$-derivations defined by
$\overline{\sigma}_i(\overline{r}):= \overline{\sigma_i(r)}$ and
$\overline{\delta}_i(\overline{r}):= \overline{\delta_i(r)}$, $1
\leq i \leq n$. Since $N(R)$ is proper, \cite{AcostaLezamaReyes}, Proposition
2.6 (ii), implies that $A/N(R)\langle x_1, \dots,
x_n \rangle$ is a skew PBW extension over $R/N(R)$. Note that $R/N(R)$
is a $\overline{\Sigma}$-rigid ring, since if
$\overline{r}\overline{\sigma}^{\alpha}(\overline{r})=0$, then
$\overline{r\sigma^{\alpha}(r)}=0$, and so $r\sigma_i(r)\in N(R)$. Having in mind that $N(R)$ is
 $\Sigma$-rigid, $r\in N(R)$ and so $\overline{r}=0$. By \cite{ReyesSuarez2018-3}, Theorem 4.4, $A/N(R)\langle x_1, \dots, x_n \rangle$ is a reduced ring (cf. \cite{Fajardoetal2020}, Theorem 6.1.9).
 
Let us see that $N(A)\subseteq N(R)\langle x_1,
\dots, x_n \rangle$. If $f\in N(A)$, then $f^k=0$, for some positive
integer $k$. Thus $\overline{f^k}=\overline{f}^k=0$
 in $A/N(R)\langle x_1, \dots, x_n \rangle$, and so $\overline{f}\in
 N\left(A/N(R)\langle x_1, \dots, x_n \rangle\right)=0$. Hence $f\in N(R)\langle x_1, \dots, x_n \rangle$, that is, $N(A)\subseteq N(R)\langle x_1, \dots, x_n \rangle$.
 
For the another inclusion, if $f=r_0+r_1X_1+\cdots + r_tX_t\in
N(R)\langle x_1, \dots, x_n \rangle$, where $r_0, r_1, \dots, r_t\in
N(R)\subseteq N(A)$, since $N(A)$ is an ideal of $A$, then $r_0,
r_1X_1, \dots, r_tX_t\in N(A)$, and
therefore $f=r_0+r_1X_1+\cdots + r_tX_t\in N(A)$.

(ii) $\Rightarrow$ (i) Let $r\in N(R)\subseteq N(A)$. Then $rx_i\in
N(R)\langle x_1, \dots, x_n \rangle$, for $1\leq i\leq n$, whence $x_ir=\sigma_i(r)x_i+\delta_i(r)\in N(R)\langle x_1, \dots, x_n
\rangle=N(A)$. By assumption $N(R)$ is a $\Sigma$-ideal, so $\sigma_i(r)\in N(R)$ and thus $\delta_i(r)\in N(R)$, i.e.,
$N(R)$ is $\Delta$-invariant. Since $N(R)$ is proper, 
\cite{AcostaLezamaReyes}, Proposition 2.6 (i), implies that $N(R)\langle x_1, \dots, x_n \rangle$ is an ideal of $A$ and $AN(R)\subseteq N(R)\langle x_1, \dots, x_n \rangle$. Using that $N(R)\langle x_1, \dots, x_n \rangle=N(A)$, we obtain that $N(A)$ is an ideal
of $A$, and hence $A$ is NI. 

To show that $N(R)$ is
$\Sigma$-rigid it is enough to see that for $r\in R$ and $1\leq i\leq
n$, $r\sigma_i(r)\in N(R)$ implies that $r\in N(R)$. If for $r\in
R$, $r\sigma_i(r)\in N(R)$, $1\leq i\leq n$, then
$r\sigma_i(r)x_i\in N(A)=N(R)\langle x_1, \dots, x_n \rangle$,
$1\leq i\leq n$, and so $\sigma_i(r)x_ir=
\sigma_i(r)\sigma_i(r)x_i+\sigma_i(r)\delta_i(r)\in N(R)\langle x_1,
\dots, x_n \rangle$. Thus $\sigma_i(r^2)\in
N(R)=\sigma_i^{-1}(N(R))$, since $N(R)$ is $\Sigma$-ideal. Then
$r^2\in N(R)$ and so $r\in N(R)$.

(i) $\Rightarrow$ (iii) If $A$ is NI then $N^{*}(A)= N(A)$, and by
Proposition \ref{prop-NIskewPBWimplRNI}, $R$ is NI, i.e.,
$N^{*}(R)=N(R)$. From implication (i) $\Rightarrow$ (ii) we have that
$N(A)=N(R)\langle x_1, \dots, x_n \rangle$, and so $N^{*}(A)= N(A)=N(R)\langle x_1, \dots, x_n \rangle=N^{*}(R)\langle x_1, \dots, x_n \rangle$.

(iii) $\Rightarrow$ (ii) Suppose that  $N(R)$ is a $\Sigma$-rigid ideal of $R$. By the same argument as in the proof of (i) $\Rightarrow$ (ii), we have that $N(R)$ is a $\Sigma$-ideal of
$R$. If $r\in N(R)$, then $x_ir=\sigma_i(r)x_i+\delta_i\in
AN(R)\subseteq N(R)\langle x_1, \dots, x_n \rangle$, and since
$\sigma_i(r)\in N(R)$ then $\delta_i(r)\in N(R)$, for $1\leq i\leq n$, which shows that $N(R)$ is $\Delta$-invariant.  Using the same argument as in
the proof of (i) $\Rightarrow$ (ii), we see that $N(A)\subseteq
N(R)\langle x_1, \dots, x_n \rangle$. Since $N(R)$ is an ideal then
$R$ is NI, and so $N^{*}(R)=N(R)$. By assumption
$N^{*}(R)\langle x_1, \dots, x_n \rangle=N^{*}(A)$, whence $N(R)\langle x_1, \dots, x_n \rangle= N^{*}(R)\langle x_1, \dots,
x_n \rangle=N^{*}(A)\subseteq N(A)$.
\end{proof}

As a particular case of Theorem \ref{teo.generteo3.1}, we have the following corollary.

\begin{corollary}[\cite{Nasr-Isf2015}, Theorem 3.1]\label{cor-genteo3.1} Let $R$ be a ring, $\sigma$ an endomorphism of $R$, and $\delta$ a $\sigma$-derivation
of $R$. Then the following statements are equivalent:
\begin{enumerate}
\item[\rm (i)] $R[x; \sigma, \delta]$ is  NI and $N(R)$ is $\sigma$-rigid.
\item[\rm (ii)] $N(R)$ is a $\sigma$-ideal of $R$ and $N(R[x; \sigma,
\delta]) = N(R)[x; \sigma, \delta]$.
\item[\rm (iii)] $N(R)$ is a $\sigma$-rigid ideal of $R$ and $N^{*}(R[x; \sigma, \delta])= N^{*}(R)[x; \sigma, \delta]$.
\end{enumerate}
\end{corollary}

\section{NJ skew PBW extensions}\label{sect-NJ}
In this section, we present the original results of the paper about NJ property for skew PBW extensions. 

We start with the following assertion that follows directly from definitions of NI and NJ rings.

\begin{proposition}\label{prop.NJskewimlRNI}
If $A = \sigma(R)\langle x_1,\dotsc, x_n\rangle$ is an NJ skew PBW extension over a ring $R$, then $R$ is NI and $J(A)\cap R= N(R)$.
\end{proposition}

From \cite{Fajardoetal2020}, Propositions 3.2.1 and 3.2.3, it follows the next result.

\begin{proposition}\label{prop-domR impl A NJ}
Skew PBW extensions over a domain $R$ are NJ rings, and hence, NI rings.
\end{proposition}

\begin{theorem}\label{teo-NJgradedSkew}
If $A$ is a graded skew PBW extension over $R=\bigoplus_{n\in
\mathbb{N}}R_n$, then $A$ is NJ if and only if $A$ is NI and
$J(A)\cap R_0$ is a nil ideal.
\end{theorem}

\begin{proof}
Suppose that $A$ is NI and $J(A)\cap R_0$ is nil. By \cite{Suarez}, Remark
2.10 (i), $R_0=A_0$, whence $J(A)\cap R_0=
J(A)\cap A_0$ is nil. From  \cite{Nasr-Isf2015}, Theorem 2.4, $J(A)=N(A)$, i.e., $A$ is an NJ ring. 

Conversely, if $A$ is an NJ ring, then $A$ is NI. Since $R_0$ is a subring,  $J(A)\cap R_0$ is an
ideal, and as $A$ is an NJ ring, $J(A)=N(A)$. In this way, $J(A)\cap R_0\subseteq J(A)=N(A)$, and hence $J(A)\cap R_0$ is a nil ideal.
\end{proof}

\begin{corollary}\label{cor-ConecNJiifNI}
If $A$ is a graded skew PBW extension over a connected algebra $R=\bigoplus_{n\in \mathbb{N}}R_n$, then $A$ is NJ if and only if $A$
is NI.
\end{corollary}

\begin{proof}
Since $R$ is connected, \cite{Suarez}, Remark 2.10 (i), implies that $A_0=R_0=\mathbb{K}$, whence $J(A)\cap A_0=\{0\}$ is nil. The result follows from Theorem
\ref{teo-NJgradedSkew}.
\end{proof}

\begin{corollary}\label{cor-QuasiNJiif NI}
Let $A$ be a quasi-commutative skew PBW extension over a ring $R$. Then $A$ is NJ if and only if $A$ is NI and $J(A)\cap R$ is a nil
ideal of $R$.
\end{corollary}

\begin{proof}
From \cite{SuarezLezamaReyes2107-1}, Proposition 2.6, we know that quasi-commutative skew PBW extensions where $R$ has the the trivial graduation are graded skew PBW extensions, so \cite{Suarez}, Remark 2.10
(i), guarantees that $A_0=R_0=R$. Thus, the assertion follows from  Theorem \ref{teo-NJgradedSkew}.
\end{proof}

\begin{corollary}\label{cor-QuasiNJimpRNJ}
If $A$ is an NJ quasi-commutative skew PBW extension over a ring $R$
and $J(R)\subseteq J(A)$ then $R$ is NJ.
\end{corollary}

\begin{proof}
By Proposition \ref{prop.NJskewimlRNI} we have that $J(A)\cap
R=N(R)$. Since $J(R)\subseteq J(A)$ then $J(R)\subseteq N(R)$. Since
$A$ is NJ then it is NI, so by Proposition
\ref{prop-NIskewPBWimplRNI} we have that $N(R)\subseteq J(R)$.
Therefore $R$ is NJ.
\end{proof}

The next theorem presents similar results to \cite{Jiang2019}, Theorem 3.10 (1).

\begin{theorem}\label{teo-genteo3.10}
Let $A=\sigma(R)\langle x_1,\dots,x_n\rangle$ be a quasi-commutative bijective skew PBW extension over $R$. If $R$ is a weakly 2-primal
weak $\Sigma$-compatible ring, then $A$ is NJ.
\end{theorem}

\begin{proof}
Since $A$ is quasi-commutative, $\delta_i=0$ for $1\leq i\leq
n$. Thus $A$ is weak
$\Delta$-compatible. If $R$ is weakly 2-primal then it is NI, and by Theorem \ref{teo- RweakCompNIiff A NI}, $A$ is NI. Let us show that $J(A)\cap R$ is nil. If $r\in J(A)\cap R$, then $rx_1\in J(A)$. From \cite{SuarezLezamaReyes2107-1}, Proposition
2.6, $A$ is a graded skew PBW extension with the trivial graduation of $R$, and so $rx_1$ is a homogeneous component of $J(A)$ with degree 1. By \cite{Nasr-Isf2015}, Lemma 2.3
(2), $rx_1\in N(A)$. In this way, \cite{ReyesSuarez2019-2}, Theorem 4.6, implies that $r\in N(R)$. The assertion follows from 
Corollary \ref{cor-QuasiNJiif NI}.
\end{proof}

\begin{remark}
Let $A=\sigma(R)\langle x_1,\dotsc, x_n\rangle$ be a skew PBW extension over a ring $R$.
\begin{enumerate}
\item [\rm (i)] If $R$ is 2-primal and $(\Sigma,
\Delta)$-compatible, then by Proposition
\ref{prop-skewPBWcompNI} (i) we have that $A$ is NI, that is,
$N^{*}(A)=N(A)$. Now, from \cite{LouzariReyes2020}, Theorem 4.11, $J(A)=N_{*}(A)$, whence $N(A)= J(A)$ and thus $A$ is NJ.
\item [\rm (ii)] If $R$ is locally finite and weak
$\Sigma$-skew Armendariz, then $R$ is NJ. More exactly, since $R$ is weak
$\Sigma$-skew Armendariz,  \cite{ReyesSuarez2018-3}, Proposition 4.9, implies that $R$ is Abelian. Now, by assumption $R$ is locally finite, so \cite{Huh2004}, Proposition 2.5, guarantees that $N(R)=J(R)$, that is, $R$ is an NJ ring.
\item [\rm (iii)] If $A$ is NJ and $J(R)$ is nil
(or $N^{*}(R)=J(R)$ or $R/N^{*}(R)$ is semiprimitive), then $R$ is
NJ. Note that if $A$ is NJ, Proposition \ref{prop.NJskewimlRNI} shows that $R$ is NI, and if $J(R)$ is nil (or $N^{*}(R)=J(R)$ or $R/N^{*}(R)$ is semiprimitive), then by \cite{Jiang2019}, Proposition 2.3, $R$ is NJ.
\end{enumerate}
\end{remark}


\begin{proposition}\label{prop-derivtypeNJ}
Let $A=\sigma(R)\langle x_1,\dots,x_n\rangle$ be a skew PBW
extension of derivation type.
\begin{enumerate}
\item[(i)] If $A$ is NI and $R$ is $\Delta$-compatible and right duo, then $A$ is
NJ.
\item[(ii)] $A$ is NI and
$I\subseteq N(R)$ if and only if $A$ is NJ, where $I\subseteq R$ is
the set of all coefficients of all terms of all polynomials of
$J(A)$.
\end{enumerate}
\end{proposition}
\begin{proof}
(i) If $A$ is NI then $N(A)=N^{*}(A)\subseteq J(A)$. For the other
inclusion, if $f=a_0+a_1X_1+\dots+a_tX_t\in J(A)$, then $fx_n=a_0x_n+a_1X_1x_n+\dots+a_tX_tx_n\in J(A)$ and therefore
$1+fx_n=1+a_0x_n+a_1X_1x_n+\dots+a_tX_tx_n$ is a unit of $A$. Since
$R$ is $(\Sigma,\Delta)$-compatible and right duo, \cite{Hamidizadehetal2020}, Theorem 4.7, guarantees that $a_0, a_1,\dots, a_t\in N(R)$. So, $f=a_0+a_1X_1+\dots+a_tX_t\in
N(R)\langle x_1,\dots,x_n\rangle$. As $A$ is NI, Proposition \ref{prop-derivtype} (iv) implies that $N(R)\langle x_1,\dots, x_n\rangle=N(A)$, that is, $f\in N(A)$.

(ii) Suppose that $A$ is NI and $I\cap R\subseteq N(R)$. By
Proposition \ref{prop-derivtype} (iv), $N(A)=N(R)\langle
x_1,\dots, x_n\rangle$. If  $f=a_0+a_1X_1+\dots+a_tX_t\in J(A)$, then by hypothesis  $a_k\in N(R)$, for $0\leq k\leq t$, whence $a_kX_k\in N(R)\langle x_1,\dots, x_n\rangle$, for every $k$. Since $A$ is NI, then $N(A)$ is an ideal of $A$, and hence $f=a_0+a_1X_1+\dots+a_tX_t\in N(A)$. 

For the converse,
suppose that $A$ is NJ. Then $A$ is NI, and so Proposition
\ref{prop-derivtype} (iv) implies that $N(A)=N(R)\langle x_1,\dots,
x_n\rangle$. Now, if $r\in I$, then for some $f=a_0+a_1X_1+\cdots
+a_tX_{t}\in J(A)$ there exists $1\leq k\leq t$ such that $r=a_k$. Since
$J(A)=N(A)=N(R)\langle x_1,\dots, x_n\rangle$, it follows that $r\in N(R)$.
\end{proof}

The following theorem is another important result of the paper.

\begin{theorem}\label{teo-NJssiNIderivtype}
Let $A=\sigma(R)\langle x_1,\dots,x_n\rangle$ be a skew PBW extension of derivation type. Then $A$ is NI if and only if $A$ is NJ.
\end{theorem}

\begin{proof}
By Proposition \ref{prop-derivtypeNJ} (ii) it is enough to prove
that every coefficient of each term  of each polynomial of $J(A)$ is
nilpotent. Let $f=r_0+r_1X_1+ \dots +r_tX_t\in J(A)$. Since $A$ is
 NI then  $N(A)$ is
an ideal of $A$, and by Proposition \ref{prop-derivtype} (iv), $N(R)$ is a $\Delta$-invariant ideal of $R$ and
$N(A)=N(R)\langle x_1,\dots, x_n\rangle$. Since $A$ is of derivation
type, $N(R)$ is a $\Sigma$-ideal. By \cite{AcostaLezamaReyes}, Proposition 2.6 (ii),  $A/N(R)\langle x_1, \dots, x_n
\rangle$ is a skew PBW extension of $R/N(R)$. By considering the notation of the proof of \cite{AcostaLezamaReyes}, Proposition 2.6 (ii), and identifying $\overline{x_i}$ with $x_i$, $1\leq i\leq n$, we use
$A/N(A)=A/N(R)\langle x_1, \dots, x_n \rangle\cong
\overline{\sigma}(R/N(R))\langle x_1, \dots, x_n \rangle$ to denote
such an extension. Now, by \cite{AcostaLezamaReyes}, Proposition 2.2 (i), the system of endomorphisms and $\Sigma$-derivations $(\Sigma, \Delta)$ induce over $R/N(R)$ a
system $(\overline{\Sigma}, \overline{\Delta})$  of endomorphisms
and $\Sigma$-derivations defined by
$\overline{\sigma}_i(\overline{r}):= \overline{\sigma_i(r)}$ and
$\overline{\delta}_i(\overline{r}):= \overline{\delta_i(r)}$, $1
\leq i \leq n$. If
$\overline{r}\overline{\sigma}_i(\overline{r})=\overline{r\sigma_i(r)}=0$
 then $r\sigma_i(r)\in N(R)$ and so $r\in N(R)$, since by Proposition \ref{prop-derivtype} (i), $N(R)$ is
 $\Sigma$-rigid. Therefore $\overline{r}=0$ and so $R/N(R)$ is $\overline{\Sigma}$-rigid. By \cite{ReyesSuarez2018-2}, Theorem 3.9, $R/N(R)\langle x_1, \dots, x_n \rangle\cong A/N(A)$ is  $(\Sigma,\Delta)$-compatible
  (see also \cite{Fajardoetal2020}, Proposition 6.2.4). Now, as $f\in J(A)$, then $\overline{f}:= f+N(A)\in J(A)/N(A)=
J\left(A/N(A)\right)$, since by Proposition
\ref{prop-NIskewPBWimplRNI} $N(A)\subseteq J(A)$. So,
$\overline{f}x_n=\overline{r_0}x_n+\overline{r_1}X_1x_n+ \dots
+\overline{r_t}X_tx_n\in J\left(\overline{\sigma}(R/N(R)\langle x_1,
\dots, x_n \rangle\right)$, with $\overline{r_k}:=r_k+N(R)$, $1\leq
k\leq n$. Having in mind that $\overline{f}x_n\in
J\left(\overline{\sigma}(R/N(R)\langle x_1, \dots, x_n
\rangle\right)$, then
$\overline{1}+\overline{f}x_n=\overline{1}+\overline{r_0}x_n+\overline{r_1}X_1x_n+
\dots +\overline{r_t}X_tx_n$ is a unit of
$\overline{\sigma}(R/N(R)\langle x_1, \dots, x_n \rangle$. Thus, by
\cite{LouzariReyes2020}, Remark 4.9 (ii),  $\overline{r_k}\in N(R/N(R))$, $0\leq k\leq n$. Since $R/N(R)$ is reduced we have that $\overline{r_k}=0$, for each $0\leq k\leq n$. In this way, $r_k\in N(R)$, for each $0\leq k\leq n$. The result follows from Proposition \ref{prop-derivtypeNJ} (ii).
\end{proof}

\begin{remark}\label{rem-igualdRadic}
Let $A=\sigma(R)\langle x_1,\dots, x_n\rangle$ be a skew PBW extension of derivation type. By Propositions \ref{prop-derivtype} and \ref{prop-derivtypeNJ}, and Theorem
\ref{teo-NJssiNIderivtype}, if  $A$ is NI then
\begin{equation*}\label{eq.igualdNINJ}
 N(A)=N^{*}(A)= N(R)\langle x_1,\dots,
x_n\rangle=N^{*}(R)\langle x_1,\dots, x_n\rangle= J(A).
\end{equation*}
\end{remark}

We immediately have the following corollary.

\begin{corollary}[\cite{Nasr-Isf2014}, Proposition 2.8]\label{cor.prop2.8} Let $R$ be a ring and $\delta$ a derivation of
$R$. If $R[x; \delta]$ is NI, then $J(R[x; \delta]) = N(R[x;
\delta])=N(R)[x; \delta]=N^{*}(R)[x; \delta]= N^{*}(R[x; \delta])$.
\end{corollary}

The next result, Proposition \ref{cor.gencor3.2}, extends \cite{Nasr-Isf2015}, Corollary 3.2.

\begin{proposition}\label{cor.gencor3.2}
If $A=\sigma(R)\langle x_1,\dots, x_n\rangle$ is a skew PBW
extension of derivation type, then the following statements are
equivalent:
\begin{enumerate}
\item[\rm (i)] $A$ is NJ.
\item[\rm (ii)] $A$ is  NI.
\item[\rm (iii)] $R$ is NI and $N(A)=N(R)\langle x_1, \dots, x_n \rangle$.
\item[\rm (iv)] $R$ is NI and $N^{*}(A)= N^{*}(R)\langle x_1, \dots, x_n \rangle$.
\end{enumerate}
\end{proposition}

\begin{proof}
(i) $\Leftrightarrow$ (ii) It follows from Theorem
\ref{teo-NJssiNIderivtype}.

(ii) $\Rightarrow$ (iii) If $A$ is NI then by Proposition \ref{prop-NIskewPBWimplRNI} we have that $R$ is NI. From Proposition \ref{prop-derivtype} (iv), $N(A)=N(R)\langle x_1, \dots, x_n \rangle$.

(iii) $\Rightarrow$ (ii) If $R$ is NI then $N(R)$ is an ideal of $R$, and by Proposition \ref{prop-derivtype} (iv), $N(R)$ is $\Delta$-invariant and therefore $A$ is NI.

(iii) $\Rightarrow$ (iv) If $R$ is NI then $N(R)$ is a
$\Sigma$-ideal. Theorem \ref{teo.generteo3.1}, part (ii) $\Rightarrow$ (iii), shows that $N^{*}(A)= N^{*}(R)\langle x_1, \dots, x_n \rangle$.

(iv) $\Rightarrow$ (iii) Since $R$ is NI, then $N(R)$ is an ideal of $R$. Now, by Proposition \ref{prop-derivtype} (i), $N(R)$ is $\Sigma$-rigid. Theorem \ref{teo.generteo3.1}, part (iii) $\Rightarrow$ (ii), implies that $N(A)= N(R)\langle x_1, \dots, x_n \rangle$.
\end{proof}

The next proposition establishes similar results to \cite{Nasr-Isf2015}, Corollary 3.3.

\begin{proposition}\label{prop-quasicomm}
If $A=\sigma(R)\langle x_1,\dots, x_n\rangle$ is a quasi-commutative skew PBW extension of $R$, then the following
statements are equivalent:
\begin{enumerate}
\item[\rm (i)] $A$ is NJ and $N(A)=N(R)\langle x_1, \dots, x_n \rangle$.
\item[\rm (ii)] $N(R)$ is a $\Sigma$-ideal of $R$ and $N(A)=N(R)\langle x_1, \dots, x_n \rangle$.
\item[\rm (iii)] $A$ is NI and $N(R)$ is $\Sigma$-rigid.
\item[\rm (iv)] $N(R)$ is $\Sigma$-rigid ideal of $R$ and $N^{*}(A)= N^{*}(R)\langle x_1, \dots, x_n \rangle$.
\end{enumerate}
\end{proposition}

\begin{proof} (i) $\Rightarrow$ (ii) If $A$ is NJ we know that $A$ is NI, and therefore $R$
is NI, so $N(R)$ is an ideal of $R$. In the proof of Theorem
\ref{teo.generteo3.1}, part (i) $\Rightarrow$ (ii), it was already shown
that $N(R)$ is a $\Sigma$-ideal.

(ii) $\Rightarrow$ (iii) Theorem
\ref{teo.generteo3.1}, part (ii) $\Rightarrow$ (i).

(iii) $\Rightarrow$ (i) If $r\in J(A)\cap R$ then $rx_1\in J(A)$. From \cite{SuarezLezamaReyes2107-1}, Proposition 2.6, $A$ is
a graded skew PBW extension with the trivial graduation of $R$. Thus $rx_1$ is a homogeneous component of $J(A)$ with grade 1. By
\cite{Nasr-Isf2015}, Lemma 2.3 (2), we have that $rx_1\in N(A)$. Since $A$ is NI and $N(R)$ is $\Sigma$-rigid, then Theorem \ref{teo.generteo3.1}, implication (i) $\Rightarrow$ (ii), guarantees that $N(A)=N(R)\langle x_1, \dots, x_n \rangle$. Since $rx_1\in N(A)$, it follows that $r\in N(R)$, and thus $J(A)\cap R$ is a nil ideal of $R$. Corollary \ref{cor-QuasiNJiif
NI} implies that $A$ is NJ.

(iii) $\Leftrightarrow$ (iv) This is precisely the content of Theorem \ref{teo.generteo3.1}, equivalence (ii) $\Leftrightarrow$ (iii).
\end{proof}

\section{Future work}
In this article, we have established necessary or sufficient conditions to ensure that skew PBW extensions are either NI or NJ rings. Now, since the notion of weak compatibility introduced in \cite{ReyesSuarez2019-2} is more general than the notion of compatibility defined in \cite{HashemiKhalilAlhevaz2017} and \cite{ReyesSuarez2018-2}, a first task is to extend the results obtained here that used the compatibility condition, now assuming the weak compatibility condition.

On the other hand, since skew PBW extensions are part of a more general family of noncommutative algebraic structures such as the {\em semi-graded rings} introduced by Lezama and Latorre \cite{LezamaLatorre2017}, which have been studied recently from the point of view of noncommutative algebraic geometry (topics such as Hilbert series and Hilbert polynomial, generalized Gelfand-Kirillov dimension, non-commutative schemes, and Serre-Artin-Zhang-Verevkin theorem, see Lezama et al., \cite{Lezama2020} and \cite{LezamaGomez2019}), and having in mind that semi-graded rings are more general than $\mathbb{N}$-graded rings (of course, we mean rings that do not admit a non-trivial $\mathbb{N}$-graduation), the natural task is to investigate NI and NJ properties for semi-graded rings with the aim of extending results established by Nasr-Isfahani \cite{Nasr-Isf2015} in the context of graded rings.

%
%



\end{document}